\documentclass{amsart}
\usepackage{amssymb}
\usepackage{amsmath}
\usepackage{amsfonts}
\usepackage{amsxtra}
\usepackage{color}
\usepackage{xcolor}
\usepackage{hyperref}

\theoremstyle{plain}
\newtheorem{theorem}{Theorem}
\newtheorem{theorem*}{Theorem}
\newtheorem{proposition}{Proposition}
\newtheorem{lemma}{Lemma}

\theoremstyle{remark}

\theoremstyle{definition}

\newcommand{\reft}[1]{Theorem \ref{#1}}
\newcommand{\refp}[1]{Proposition \ref{#1}}
\newcommand{\refl}[1]{Lemma \ref{#1}}
\newcommand{\refs}[1]{Section \ref{#1}}
\newcommand{\refeq}[1]{Eq. (\ref{#1})}

\def\sig{\sigma}
\def\vsig{\zeta}
\def\tet{\xi}

\def\k{F}
\def\C{\mathbb{C}}
\def\N{\mathbb{N}}
\def\R{\mathbb{R}}

\def\SS{\mathbb{S}}

\def\cx{\C^{\times}}

\def\ka{\k^{+}}
\def\kc{\k^{\circ}}

\def\VC{\mathcal{V}}
\def\WC{\mathcal{W}}

\def\fru{\mathfrak{u}}

\def\G{\mathbf{G}}

\def\hG{\widehat{G}}

\def\gs{g^{\sig}}

\def\ks{\k^{\sig}}
\def\Gs{C_{G}(\sig)}
\def\Hs{C_{H}(\sig)}
\def\Ms{C_{M}(\sig)}
\def\Ns{C_{N}(\sig)}
\def\Js{C_{J}(\sig)}
\def\CG{\C[G]}
\def\CH{\C[H]}
\def\CGs{\C[\Gs]}

\def\CMs{\C[\Ms]}
\def\CNs{\C[\Ns]}

\newcommand{\map}[3]{#1 \colon #2 \to #3}

\def\inv{^{-1}}
\def\x{\times}
\def\ox{\otimes}

\def\Res{\operatorname{Res}}
\def\Ind{\operatorname{Ind}}
\def\CInd{\operatorname{c-Ind}}

\def\spec{\operatorname{Spec}}
\def\ab{\mathrm{ab}}

\def\id{\mathrm{id}}

\allowdisplaybreaks

\begin{document}

\title[]{Smooth representations of involutive algebra groups over non-archimedean local fields}

\author[]{Carlos A. M. Andr\'e}

\author[]{Jo\~ ao Dias}

\thanks{This research was made within the activities of the Group for Linear, Algebraic and Combinatorial Structures of the Center for Functional Analysis, Linear Structures and Applications (University of Lisbon, Portugal), and was partially supported by the Portuguese Science Foundation (FCT) through the Strategic Project UID/MAT/04721/2013. The second author was partially supported by the Lisbon Mathematics PhD program (funded by the Portuguese Science Foundation). This work is part of the second author Ph.D. thesis.}

\address[C. A. M. Andr\'e]{Centro de An\'alise Funcional, Estruturas Lineares e Aplica\c c\~oes (Grupo de Estruturas Lineares e Combinat\'orias) \\ Departamento de Matem\'atica \\ Faculdade de Ci\^encias da Universidade de Lisboa \\ Campo Grande, Edif\'\i cio C6, Piso 2 \\ 1749-016 Lisboa \\ Portugal}
\email{caandre@ciencias.ulisboa.pt}

\address[J. Dias]{Centro de An\'alise Funcional, Estruturas Lineares e Aplica\c c\~oes (Grupo de Estruturas Lineares e Combinat\'orias) \\ Departamento de Matem\'atica \\ Faculdade de Ci\^encias da Universidade de Lisboa \\ Campo Grande, Edif\'\i cio C6, Piso 2 \\ 1749-016 Lisboa \\ Portugal}
\email{joaodias104@gmail.com}

\subjclass[2010]{20G25; 22D12; 22D30}

\date{\today}

\keywords{Algebra with involution; algebra group; fixed point subgroup; smooth representation; induction with compact support; Gutkin's conjecture}

\begin{abstract}
An algebra group over a field $\k$ is a group of the form $G = 1+J$ where $J$ is a finite-dimensional nilpotent associative $\k$-algebra. A theorem of M. Boyarchenko asserts that, in the case where $\k$ is a non-archimedean local field, every irreducible smooth representation of $G$ is admissible and smoothly induced by a one-dimensional smooth representation of some algebra subgroup of $G$. If $J$ is a nilpotent algebra endowed with an involution $\map{\sig}{J}{J}$, then $\sig$ naturally defines a group automorphism of $G$, and we may consider the fixed point subgroup $\Gs$. Assuming that $F$ has characteristic different from $2$, we extend Boyarchenko's result and show that every irreducible smooth representation of $\Gs$ is admissible and smoothly induced by a one-dimensional smooth representation of a subgroup of the form $\Hs$ where $H$ is an $\sig$-invariant algebra subgroup of $G$. As a particular case, the result holds for maximal unipotent subgroups of the classical Chevalley groups defined over $\k$.
\end{abstract}

\maketitle

\section{Introduction}

Throughout the paper, unless otherwise stated, $\k$ will always denote a local field, that is, a non-discrete locally compact topological field. It is well-known (see \cite[Theorem~II.3]{Weil1974a}) that $\k$ is \textit{self-dual} in the sense that, if we fix a nontrivial unitary character $\map{\tet}{\ka}{\cx}$ of the additive group $\ka$ of $\k$ and define $\map{\tet_{a}}{\ka}{\cx}$ by $\tet_{a}(x) = \tet(ax)$ for all $x \in \k$, then the mapping $a \mapsto \tet_{a}$ defines a topological isomorphism between $\ka$ and its Pontryagin dual $\kc$.

Let $A$ be a finite-dimensional associative $\k$-algebra with identity, and let $J$ be a nilpotent subalgebra of $A$. Then, $G = 1+J$ is a subgroup of the unit group of $A$. Following \cite{Isaacs1995a}, a group $G$ constructed in this way will be referred to as an \textit{algebra group} over $\k$. As a typical example, if $J = \fru_{n}(\k)$ is the $\k$-algebra consisting of all strictly upper triangular $n \x n$ matrices over $\k$, then the corresponding algebra group $G = 1+J$ is isomorphic to the upper unitriangular group $U_{n}(\k)$.

The topology of $\k$ induces naturally a topology on $J$, and hence also a topology on $G = 1+J$ with respect to which $G$ becomes a locally compact and second countable topological group. Moreover, $G$ is unimodular, and every $\k$-subalgebra $L$ of $J$ defines a closed subgroup $H = 1+L$ of $G$. We refer to such a subgroup as an \textit{algebra subgroup} of $G$.

We will always assume that $A$ is equipped with an involution $\map{\sig}{A}{A}$, in the sense that the following conditions are satisfied for all $a, b \in A$:  $\sig(a+b)=\sig(a)+\sig(b)$, $\sig(ab)=\sig(b)\sig(a)$, and $\sig^2=\id_{A}$. The involution $\sig$ is not required to be $\k$-linear. However, we will assume that the field $\k = \k\cdot 1$ is preserved by $\sig$. Then, $\sig$ defines a field automorphism of $\k$ which is either the identity or has order $2$. In any case, we let $\ks$ denote the $\sig$-fixed subfield of $\k$, and consider that $A$ is a finite-dimensional associative $\ks$-algebra. If $J$ is a $\sig$-invariant nilpotent subalgebra of $A$, then the cyclic group $\langle \sig \rangle$ acts on the algebra group $G = 1+J$ as a group of automorphisms by means of $\gs = \sig(g\inv)$ for all $g \in G$. We will denote by $\Gs$ the subgroup of $G$ consisting of all $\sig$-fixed elements, that is, $$\Gs = \{g\in G \colon \gs = g\} = \{g \in G \colon \sig(g\inv) = g\}.$$ It is clear that $\Gs$ is a closed subgroup of $G$, but not necessarily an algebra subgroup.

Our main result is about smooth representations of groups of the form $\Gs$. For simplicity of writing, we prefer to use the equivalent notion of smooth modules over the group algebra of $\Gs$. More generally, let $G$ be an arbitrary topological group, and let $\CG$ denote the complex group algebra of $G$. Then, a left $\CG$-module $\VC$ is said to be \textit{smooth} if, for every $v \in \VC$, the stabiliser $G_{v} = \{g \in G \colon gv = v\}$ is an open subgroup of $G$. In the particular case where $\VC$ is one-dimensional, we naturally obtain a group homomorphism $\map{\tet}{G}{\cx}$ with open kernel. We will refer to such a homomorphism as a \textit{smooth character} of $G$, and denote by $G^{\circ}$ the set consisting of all smooth characters of $G$. Throughout the paper, for every smooth character $\tet \in G^{\circ}$, we will denote by $\C_{\tet}$ the one-dimensional smooth $\CG$-module whose underlying vector space is $\C$ and where the $G$-action is given by $g\alpha = \tet(g)\alpha$ for all $g \in G$ and all $\alpha \in \C$. It is well-known that $G^{\circ}$ is a group with respect to the usual multiplication of characters. It should not be confused with the \textit{Pontryagin dual} $\hG$ of $G$ which consists of all unitary characters $\map{\tet}{G}{\cx}$ of $G$. By definition, a \textit{unitary character} of $G$ is a continuous group homomorphism $\map{\tet}{G}{\cx}$ whose image $\tet(G)$ lies inside the unit circle $\SS^{1}$ in $\C$.

We are mainly concerned with the case where $\k$ is a non-archimedean local field. In this situation, every algebra group $G = 1+J$ over $\k$ is an $\ell$-group. Following the terminology of \cite{Bernshtein1976a}, by an \textit{$\ell$-group} we mean a topological group whose identity has a neighbourhood basis consisting of compact open subgroups. In fact, $G = \G(\k)$ is the subgroup consisting of all $\k$-rational points of the connected linear algebraic group $\G = 1+(J\ox_{\k} \overline{\k})$ where $\overline{\k}$ denotes the algebraic closure of $\k$. Furthermore, since $\G$ is unipotent, the group $G = \G(\k)$ is the filtered union of a family of compact open subgroups, in the sense that every element of $G$ is contained in a compact open subgroup, and any two such subgroups are contained in a third such subgroup. Following \cite{Boyarchenko2011a}, we refer to a topological group satisfying this property as an \textit{$\ell_{c}$-group}. We note that every unitary character of an $\ell$-group is a smooth character, but the converse is not necessarily true. However, the equality $G^{\circ} = \hG$ holds for every $\ell_{c}$-group $G$; see \cite[Proposition~1.6]{Bushnell2006a}.

In \cite{Gutkin1974a}, E. Gutkin claimed that, if $G$ is an algebra group over a self-dual local field $\k$, then every irreducible unitary representation of $G$ is induced (in the sense of Mackey) from a unitary character of some algebra subgroup of $G$. In \cite{Isaacs1995a}, I.M. Isaacs presented a counterexample to Gutkin's original proof, but a complete proof was provided by Boyarchenko in \cite{Boyarchenko2011a}. We should mention that, in the case where $\k$ has characteristic zero, Gutkin's theorem may be proved using the method of coadjoint orbits, introduced by A.A. Kirillov in \cite{Kirillov1962a} for unipotent groups over $\R$ or $\C$, and adapted for unipotent groups over a $p$-adic field by C. Moore in \cite{Moore1965a}. In the case where $\k$ is a non-archimedean local field (of arbitrary characteristic), Gutkin's theorem can be formulated in the setting of smooth representations of algebra groups over $\k$. In fact, M. Boyarchenko proved the following result; see \cite[Theorem~1.3]{Boyarchenko2011a}.

\begin{theorem*}[Boyarchenko]
Let $\k$ be a non-archimedean local field, let $G$ be an algebra group over $\k$, and let $\VC$ be an irreducible smooth $\CG$-module. Then, $\VC$ is admissible, there exist an algebra subgroup $H$ of $G$ and a smooth character $\map{\tet}{H}{\cx}$ such that $\VC \cong \CInd^{G}_{H}(\C_{\tet})$. Furthermore, $\VC$ is admissible, and thus $\CInd^{G}_{H}(\C_{\tet}) = \Ind^{G}_{H}(\C_{\tet})$.
\end{theorem*}

Here, and henceforth, if $H$ is a closed subgroup of $G$ and $\WC$ is a smooth $\CH$-module, then we denote by $\Ind^{G}_{H}(\WC)$ the \textit{smoothly induced} $\CG$-module, and by $\CInd^{G}_{H}(\WC)$ the \textit{compactly induced} $\CG$-module. The definitions and basic properties of these induction functors can be found in \cite[Sections~2.4-2.5]{Bushnell2006a} or \cite[Chapitre~I,~Section~5.1]{Vigneras1996a}. Both $\Ind^{G}_{H}(\WC)$ and $\CInd^{G}_{H}(\WC)$ are smooth $\CG$-modules, and $\CInd^{G}_{H}(\WC)$ is in fact a $\CG$-submodule of $\Ind^{G}_{H}(\WC)$. In general, $\CInd^{G}_{H}(\WC)$ is a proper $\CG$-submodule of $\Ind^{G}_{H}(\WC)$, but equality holds whenever $\WC$ and, either $\CInd^{G}_{H}(\WC)$, or $\Ind^{G}_{H}(\WC)$, are \textit{admissible}; see \cite[Section~5.6,~pg.~42]{Vigneras1996a}. We recall that a smooth $\CG$-module $\VC$ is said to be \textit{admissible} if, for every compact open subgroup $K$ of $G$, the vector subspace $\VC^{K}$ consisting of all $K$-fixed vectors is finite-dimensional.

The main purpose of this paper is to prove the following result.

\begin{theorem} \label{main1}
Let $\k$ be a non-archimedean local field of characteristic different from $2$, and let $A$ be a finite-dimensional $\k$-algebra equipped with an involution $\sig$. Let $J$ be a $\sig$-invariant nilpotent subalgebra of $A$, let $G = 1+J$, and let $\VC$ be an irreducible smooth $\CGs$-module. Then, there exist a $\sig$-invariant algebra subgroup $H$ of $G$ and a smooth character $\map{\vartheta}{\Hs}{\cx}$ such that $\VC \cong \CInd^{\Gs}_{\Hs}(\C_{\vartheta})$. Furthermore, $\VC$ is admissible, and thus $\CInd^{\Gs}_{\Hs}(\C_{\vartheta}) = \Ind^{\Gs}_{\Hs}(\C_{\vartheta})$.
\end{theorem}

The proof will be by induction on the dimension of $J$, the result being obvious in the case where $\VC$ is one-dimensional. Thus, in the following section, we will assume that $\dim \VC \geq 2$, and construct an adequate $\sig$-invariant subalgebra $J_{0}$ of $J$ with codimension $1$ (which depends on the choice of $\VC$). We note that every subalgebra $J_{0}$ of $J$ with codimension $1$ contains $J^{2}$, and hence is a two-sided ideal of $J$. Otherwise, $J = J_{0}+J^{2}$, and this implies that $J_{0} = J$; for a proof see, for example, \cite[Lemma~3.1]{Isaacs1995a}. Then, $G_{0} = 1+J_{0}$ is a normal closed subgroup of $G$, $C_{G_{0}}(\sig)$ is a normal closed subgroup of $\Gs$, and we may consider the restriction $\Res^{\Gs}_{C_{G_{0}}(\sig)}(\VC)$ of $\VC$ to $C_{G_{0}}(\sig)$.

\section{The inductive step} \label{step}

Throughout this section, we let the notation be as described in \reft{main1},  and assume that there exists an irreducible smooth $\CGs$-module $\VC$ with $\dim \VC \geq 2$. In particular, by Schur's lemma (see \cite[pg.~21]{Bushnell2006a}), the group $\Gs$ cannot be abelian. Therefore, the algebra group $G = 1+J$ is also not abelian, and thus $J^{2} \neq 0$.

One crucial tool to deal with groups of $\sig$-fixed elements is the \textit{Cayley transform} $\map{\Psi}{J}{G}$ defined by $\Psi(a)=(1-a)(1+a)\inv$ for all $a\in J$. We note that, $\Psi(a) = 1-2a+2a^2-2a^3+\cdots$ for all $a\in J$. Under our assumption that $\k$ has characteristic not equal to $2$, $\Psi$ is a bijection with inverse $\map{\Phi}{G}{J}$ given by
$\Phi(x)=(x-1)(x+1)\inv$ for all $x\in G$. As a first application, we obtain the following result.

\begin{lemma} \label{cayley}
Let $\Js = \{a\in J \colon \sig(a) = -a\}$. Then, the Cayley transform defines a bijection between $\Js$ and $\Gs$. Furthermore, $\Js$ is a Lie subalgebra of $J$ when $J$ is considered as a $\ks$-algebra. 
\end{lemma} 

\begin{proof}
We note that $\Psi(-a)=\Psi(a)\inv$ for all $a \in J$, and that $\Phi(x\inv) = -\Phi(x)$ for all $x\in G$. Therefore, we deduce that $\sig(\Psi(a)) = \Psi(\sig(a)) = \Psi(-a) = \Psi(a)\inv$ for all $a\in \Js$. A similar calculation for $\Phi$ shows that $\Psi$ is indeed a bijection between $\Js$ and $\Gs$. It is also easy to check that $[a,b] \in \Js$ for all $a, b \in \Js$, and so $\Js$ is a Lie subalgebra of $J$ (when $J$ is considered as a $\ks$-algebra).
\end{proof}

The proof of \reft{main1} will proceed by induction on the dimension of $J$. To start with, we consider the descending sequence $J \supseteq J^{2} \supseteq J^{3} \supseteq \cdots$ of two-sided ideals of $J$. For every $n \in \N$, we set $G_{n} = 1+J^{n}$, so that we obtain a descending sequence $$G = G_{1} \supseteq G_{2} \supseteq G_{3}\supseteq\cdots$$ of normal subgroups of $G$. It is obvious that, for every $n \in \N$, the ideal $J^{n}$ is $\sig$-invariant, and hence the subgroup $G_{n}$ is also $\sig$-invariant. Therefore, we obtain a descending sequence $$\Gs = C_{G_{1}}(\sig) \supseteq C_{G_{2}}(\sig) \supseteq C_{G_{3}}(\sig) \supseteq\cdots$$ of normal subgroups of $\Gs$. We now prove the following auxiliary result.

\begin{lemma} \label{comm}
For every $n \in \N$, we have $[G,G_{n}] \cap \Gs = [\Gs,C_{G_{n}}(\sig)]$.
\end{lemma} 

\begin{proof} 
Let $[G,\sig]$ denote the subgroup of $G$ generated by all the elements of the form $g\inv \gs$ for $g\in G$. Since $\gs = \sig(g\inv)$, we see that $[G,\sig]$ is also generated by the set $\{g\,\sig(g) \colon g \in G\}$. Now, $G$ decomposes as the product $G = \Gs\,[G,\sig]$, and we clearly have $\Gs \cap [G,\sig] = 1$. Moreover, for every $g \in G$ and every $h \in \Gs$, we have $$h(g\,\sig(g))h\inv = h(g\,\sig(g))\,\sig(h) = (hg)\,\sig(hg),$$ which implies that $[G,\sig]$ is a normal subgroup of $G$. Since this argument does not depend on $G$, we obtain a similar decomposition $G_{n} = C_{G_{n}}(\sig)\,[G_n,\sig]$ for all $n\in \N$. 

Let $n \in \N$ be arbitrary, let $g\in G$, and let $h\in G_n$. Then, $g = g_1g_2$ for uniquely determined elements $g_1 \in \Gs$ and $g_2\in [G,\sig]$. Similarly, $h = h_{1}h_{2}$ for uniquely determined elements $h_1 \in C_{G_{n}}(\sig)$ and $h_2\in [G_n,\sig]$. Therefore, we get $$ghg\inv h\inv = (g_1h_1g_1\inv h_1 \inv) (h_1(g_1((h_1\inv g_2 h_1)h_2 g_2)g_1 \inv )h_1\inv ).$$ Since $$h_1(g_1((h_1\inv g_2 h_1)h_2 g_2)g_1 \inv )h_1\inv\in [G,\sig],$$ we conclude that $ghg\inv h\inv\in \Gs$ if and only if $ghg\inv h\inv = g_1h_1g_1\inv h_1\inv$. Using a recursive argument, we see that the analogous conclusion holds for an arbitrary product of commutators, and this completes the proof.
\end{proof}

Since $J$ is nilpotent, there is $n \in \N$ be such that $J^{n} \neq 0$ and $J^{n+1} = 0$. As justified above, we must have $n \geq 2$. Since $C_{G_{n}}(\sig)$ lies in the centre of $\Gs$, Schur's lemma implies that $C_{G_{n}}(\sig)$ acts on $\VC$ by scalar multiplications. Thus, we may choose the smallest positive integer $m$ for which there exists a smooth character $\map{\vsig}{C_{G_{m}}(\sig)}{\cx}$ such that
\begin{equation} \label{eq1}
gv=\vsig(g)v \tag{$\dag$}
\end{equation}
for all $g\in C_{G_{m}}(\sig)$ and all $v\in \VC$. We note that, since $\VC$ is an irreducible smooth $\CGs$-module with $\dim \VC \geq 2$, we must have $m \geq 2$. Furthermore, since\linebreak $[\Gs, C_{G_{m-1}}(\sig)] \subseteq C_{G_{m}}(\sig)$, the minimal choice of $m$ implies that $\vsig$ is not identically equal to $1$. Otherwise, Schur's lemma would imply that $C_{G_{m}}(\sig)$ acts on $\VC$ by scalar multiplications.

Our next goal is to prove that there exists a $\sig$-invariant ideal $L$ of $J$ satisfying $$J^{m} \subseteq L \subseteq J^{m-1},\quad \dim (L \slash J^{m})=1,\quad \text{and}\quad \vsig([\Gs, C_{1+L}(\sig)]) \neq 1.$$ We note that, in particular, $C_{G_{m}}(\sig) \subseteq  C_{1+L}(\sig) \subseteq C_{G_{m-1}}(\sig)$. To achieve this, we first prove the following result.

\begin{lemma} \label{involutionisom}
For every $m \in \N$ with $m \geq 2$, there is an isomorphism of abelian groups $C_{G_{m-1}}(\sig) \slash C_{G_{m}}(\sig) \cong C_{J^{m-1}}(\sig) \slash C_{J^{m}}(\sig)$.
\end{lemma}

\begin{proof}
Firstly, we observe that the mapping $u \mapsto 1 + u$ defines an isomorphism of abelian groups $J^{m-1}\slash J^{m} \cong 1+(J^{m-1}\slash J^{m})$, and thus we naturally obtain an isomorphism of abelian groups $J^{m-1}\slash J^{m} \cong G_{m-1} \slash G_m$. It is straightforward to check that this isomorphism is $\sig$-invariant, and that it restricts to an isomorphism $$(C_{J^{m-1}}(\sig)+J^{m})\slash J^{m} \cong (C_{G_{m-1}}(\sig) G_m)\slash G_m.$$ The result follows.
\end{proof}

By the minimal choice of $m$, we know that $C_{G_{m-1}}(\sig)\slash C_{G_{m}}(\sig) \neq 1$, and thus $C_{J^{m-1}}(\sig) \slash C_{J^{m}}(\sig) \neq 0$. Since both $C_{J^{m-1}}(\sig)$ and $C_{J^{m}}(\sig)$ are $\ks$-vector subspaces of $\Js$ (by \refl{cayley}), we have $$C_{J^{m-1}}(\sig) = L_1 \oplus \cdots \oplus L_t$$ where $L_1, \ldots, L_t$ are $\ks$-vector subspaces of $\Js$ satisfying $$C_{J^{m}}(\sig) \subseteq L_{i} \subseteq C_{J^{m-1}}(\sig) \quad \text{and} \quad \dim L_i = \dim C_{J^{m}}(\sig) + 1$$ for all $1 \leq i \leq t$. By the isomorphism above, we conclude that $$C_{G_{m-1}}(\sig) = (1+L_1) \cdots (1+L_t),$$ and thus there exists $1 \leq s \leq t$ such that $\vsig([\Gs,1+L_{s}]) \neq 1$. Otherwise, we would have $\vsig([C_{G_{m-1}}(\sig),C_{G_{m-1}}(\sig)]) = 1$, and hence $C_{G_{m-1}}(\sig)$ would act on $\VC$ by scalar multiplications.

Let $u \in L_{s}$ be such that $L_{s} = \ks u +C_{J^{m}}(\sig)$, and define $L = \k u + J^{m}$. We note that $C_{L}(\sig) = L_{s}$. We set $N = 1+L$, and observe that $N$ is a normal subgroup of $G$, and thus $\Ns$ is a normal subgroup of $\Gs$. It is clear that $$N = (1+\k u) M \quad \text{and} \quad \Ns = (1+\ks u)\Ms$$ where we set $M = G_{m}$. We note that the smooth character $\map{\vsig}{\Ms}{\cx}$ is $\Gs$-invariant. Indeed, \refl{comm} asserts that $$[\Gs,\Ms] = [G,M]\cap \Gs \subseteq C_{G_{m+1}}(\sig),$$ and so $\vsig([\Gs,\Ms])=1$.

\begin{proposition}\label{invphi}
Let $\map{\vsig}{\Ms}{\cx}$ be a $\Gs$-invariant smooth character of $\Ms$, and define $$S = \{a\in \Js \colon \vsig([\Psi(a),\Psi(b)]) = 1 \text{ for all } b\in C_{L}(\sig)\}$$ where $\map{\Psi}{J}{G}$ is the Cayley transform. Then, $S$ is a $\ks$-vector subspace of $\Js$ satisfying $C_{J^{2}}(\sig) \subseteq S$ and $\dim S \geq \dim \Js-1$. Furthermore, if we define the map $\map{\varphi}{\Gs}{\Ns^{\circ}}$ by $$\varphi(g)(h) = \vsig([g,h])$$ for all $g\in \Gs$ and all $h\in \Ns$, then $\varphi$ is a group homomorphism satisfying $$\ker(\varphi) = \Psi(S)\quad \text{and}\quad \varphi(\Gs)\subseteq \Ms^{\perp}$$ where $\Ms^{\perp}$ is the orthogonal subgroup of $\Ms$ in $\Ns^{\circ}$. In particular, $\varphi$ induces naturally a group homomorphism $\map{\widehat{\varphi}}{\Gs}{(\Ns\slash \Ms)^{\circ}}$.
\end{proposition}

\begin{proof}
We first observe that the map $\varphi$ is a well-defined group homomorphism. On the one hand, we have $$[\Gs,\Ns] \subseteq [\Gs,C_{G_{m-1}}(\sig)] \subseteq \Ms.$$ On the other hand, since $[g,hk]=[g,k][g,h]^k$ and since $\vsig$ is $\Gs$-invariant, we deduce that $$\varphi(g)(hk) = \vsig([g,k])\vsig([g,h]) = \varphi(g)(h)\varphi(g)(k)$$ for all $g\in \Gs$ and all $h,k \in \Ns$. It follows that, for every $g\in \Gs$, the map $\map{\varphi(g)}{\Ns}{\cx}$ is a smooth character of $\Ns$. Similarly, since $[gh,k]=[g,k]^h[h,k]$, we have $$\varphi(gh)(k) = \vsig([g,k])\vsig([h,k]) = \varphi(g)(k)\varphi(h)(k)$$ for all $g\in \Gs$ and all $h,k \in \Ns$, and thus $\map{\varphi}{\Gs}{\Ns^{\circ}}$ is a group homomorphism.

Now, since $[\Gs,\Ms]\subseteq \ker(\vsig)$ (because $\vsig$ is $\Gs$-invariant), the image $\varphi(\Gs)$ lies in $\Ms^{\perp}$. Moreover, it is clear by the definition that $\ker(\varphi)=\Psi(S)$. Let $a\in \Js$, and let $\alpha \in \ks$ be arbitrary. We claim that $$[\Psi(\alpha a),\Psi(b)]\,[\Psi(a),\Psi(\alpha b)]\inv \in C_{G_{m-1}}(\sig)$$ for all $b\in C_{J^{m-1}}(\sig)$. Indeed, let $b\in C_{J^{m-1}}(\sig)$ be arbitrary. Then, \cite[Proposition 3.1]{Boyarchenko2011a} implies that $$[\Psi(\alpha a),\Psi(b)]\,[\Psi(a),\Psi(\alpha b)]\inv\in [G,M],$$ and thus it follows from \refl{comm} that $$[\Psi(\alpha a),\Psi(b)]\,[\Psi(a),\Psi(\alpha b)]\inv \in [\Gs,\Ms].$$ Since $\vsig$ is $\Gs$-invariant, we conclude that
\begin{equation}\label{eq2}
\vsig([\Psi(\alpha a),\Psi(b)]) = \vsig([\Psi(a),\Psi(\alpha b)])\tag{$\ddagger$}
\end{equation}
for all $b\in C_{J^{m-1}}(\sig)$, and this implies that $\alpha a\in S$ for all $\alpha\in \ks$ and all $a \in S$.

On the other hand, \cite[Theorem 1.4]{ Halasi2004a} asserts that $[G_{2},N]\subseteq [G_{2},G_{m-1}]\subseteq [G,M]$, and thus $$[C_{G_{2}}(\sig),\Ns]\subseteq \Gs \cap [G,M] \subseteq [\Gs,\Ms].$$ Since $\Ns = \Psi(C_{L}(\sig))$, we conclude that $C_{J^{2}}(\sig) \subseteq S$. Since $\ker(\varphi) = \Psi(S)$ and since $$\Psi(a+b)\inv \Psi(a)\Psi(b)\in C_{G_{2}}(\sig) \subseteq \Psi(S),$$ we see that $\Psi(a+b) \in \ker(\varphi)$ for all $a,b \in S$. It follows that $S$ is a $\ks$-vector subspace of $\Js$ satisfying $C_{J^{2}}(\sig) \subseteq S$.

To conclude the proof, we observe that there are canonical isomorphisms of groups $$\Ms^{\perp} \cong (\Ns\slash \Ms)^{\circ} \quad \text{and} \quad \Ns\slash \Ms \cong \Psi(L \slash J^{m}) \cong \ks.$$ Since $\ks$ is a self-dual field and since $$\Gs\slash \ker(\varphi) \cong \varphi(\Gs) \subseteq \Ms^{\perp},$$ it follows that $\dim \Js - \dim S \leq 1$, as required.
\end{proof}

Next, we prove the following crucial result.

\begin{proposition} \label{invextension}
Let $\map{\vsig}{\Ms}{\cx}$ be a $\Gs$-invariant smooth character of $\Ms$, and define $S \subseteq \Js$ as in \refp{invphi}. Then, $[\Ns,\Ns] \subseteq \ker(\vsig)$, and there exists $\tet \in \Ns^{\circ}$ such that $\tet\vert_{\Ms} = \vsig$. Moreover, the following properties hold.
\begin{enumerate}
\item If $\tet' \in \Ns^{\circ}$ is such that $\tet'\vert_{\Ms} = \vsig$, then $C_{\Gs}(\tet') = \Psi(S)$.
\item If $C_{\Gs}(\tet) \neq \Gs$ and if $\tet' \in \Ns^{\circ}$ is such that $\tet'\vert_{\Ms} = \vsig$, then there exists $g \in \Gs$ such that $\tet' = \tet^{g}$.
\end{enumerate}
\end{proposition}

\begin{proof}
By construction, we have $C_{L}(\sig)=\ks u \oplus C_{J^{m}}(\sig)$ for some $u\in C_{L}(\sig)$, and hence $$\Ns=(1+\ks u)\Ms = \Psi(\ks u) \Ms.$$ Since $[\Psi(\alpha u),\Psi(\beta u)]=1$, it is obvious that $\vsig([\Psi(\alpha u),\Psi(\beta u)])=1$ for all $\alpha,\beta\in \ks$. Since $\vsig$ is $\Gs$-invariant, it follows that $\vsig([\Ns,\Ns]) = 1$. Let $\mathbb{C}_\vsig$ denote the canonical one-dimensional $\CMs$-module associated with $\vsig$, and let $\WC$ be an irreducible subquotient of the smooth $\CNs$-module $\Ind^{\Ns}_{\Ms}(\C_\vsig)$. Since $\Ms$ is a normal subgroup of $\Ns$ and since $\vsig$ is $\Ns$-invariant, we have $x\phi = \vsig(x)\phi$ for all $x\in \Ms$ and all $\phi \in \Ind^{\Ns}_{\Ms}(\C_\vsig)$. Therefore, we must have $xw = \vsig(x)w$ for all $x\in \Ms$ and all $w\in \WC$. Since $[\Ns,\Ns]\subseteq \ker(\vsig)$, Schur's lemma implies that $\dim \WC = 1$, and thus $\WC$ affords a smooth character $\map{\tet}{\Ns}{\cx}$ which clearly satisfies $\tet\vert_{\Ms}=\vsig$.

Next, we consider the group homomorphism $\varphi:\Gs\rightarrow \Ns^{\circ}$ as defined in Proposition \ref{invphi}. We recall that $\varphi(\Gs)\subseteq \Ms^{\perp}$ and that $\ker(\varphi)=\Psi(S)$ where $S$ is a $\ks$-vector subspace of $\Js$ satisfying $C_{J^{2}}(\sig)\subseteq S$ and $\dim S\geq \dim \Js -1$. On the one hand, (i) follows because $$C_{\Gs}(\tet') = \ker(\varphi) = \Psi(S)$$ for all $\tet'\in \Ns^{\circ}$ such that $\tet'\vert_{\Ms} =\vsig$.

On the other hand, assume that $C_{\Gs}(\tet)\neq \Gs$, so that $\ker(\varphi)\neq \Gs$ and $S\neq \Js$. Let $x\in \Gs$ be such that $\varphi(x)\in \Ns^{\circ}$ is not trivial, and let $a\in \Js$ be such that $x=\Psi(a)$. Then, \refeq{eq2} implies that $\varphi(\Psi(\alpha a)) \in \varphi(\Gs) = \Ms^{\perp}$ for all $\alpha\in\ks$. Since $\Ms^{\perp}\cong (\Ns \slash \Ms)^{\circ}$ and since $$\Ns \slash \Ms\cong 1+(L \slash J^{m})\cong \ks$$ (see \refl{involutionisom}), it is straightforward to show that the mapping $\alpha\mapsto \varphi(\Psi(\alpha a))$ defines a isomorphism of groups $\ks \cong \Ms^{\perp}$. In particular, it follows that $$\Ms^{\perp} = \{\varphi(\Psi(\alpha a)) \colon \alpha\in\ks\},$$ and so the map $\map{\varphi}{\Gs}{\Ms^{\perp}}$ is surjective. Therefore, we see that there are isomorphisms of groups $$\Gs\slash C_{\Gs}(\tet) \cong \Ms^{\perp} \cong \ks.$$

To conclude the proof of (ii), let $\tet'\in \Ns^{\circ}$ be such that $\tet'\vert_{\Ms}=\vsig$, and consider the character $\tet'\tet\inv\in \Ns^{\circ}$. It is obvious that $\tet'\tet\inv\in \Ms^{\perp}$, and thus there exists $\alpha\in\k$ such that $\tet'\tet\inv = \varphi(\Psi(\alpha a))$. If we set $g = \Psi(\alpha a)\inv$, then
\begin{align*}
\tet'(x)\tet(x)\inv&=\vsig ([g\inv,x\inv])=\vsig(gxg\inv x\inv)\\
&=\tet(gxg\inv x\inv)=\tet(gxg\inv)\tet(x)
\end{align*}
and so $\tet'(x)=\tet(gxg\inv)$ for all $x\in \Ns$, as required.
\end{proof}


\section{Proof of \reft{main1}} \label{proof}

In this section, we complete the proof of \reft{main1}. Hence, we keep the notation of that theorem. In particular, $\k$ will be a non-archimedean local field $\k$ of characteristic different from $2$.

We start by recalling some general notions of the theory of smooth representations of algebra groups. Let $N$ be an arbitrary $\sig$-invariant algebra subgroup of $G$, and let $\WC$ be an arbitrary smooth $\CNs$-module. For every smooth character $\tet \in \Ns^{\circ}$, let $\WC(\tet)$ be the $\k$-linear span of set $\{xw-\tet(x)w \colon x \in \Ns, w \in \WC\}$, and consider the quotient $\WC_{\tet} = \WC\slash \WC(\tet)$. Therefore, $\WC_{\tet}$ is the largest quotient of $\WC$ where $\Ns$ acts via the character $\tet$. If $\VC$ is an arbitrary smooth $\CGs$-module, then $\VC$ is also a smooth $\CNs$-module, and thus we may consider the quotient $\VC_{\tet} = \VC\slash \VC(\tet)$. We define the \textit{spectral support of $\VC$ with respect to $\Ns$} to be the subset $$\spec_{\Ns}(\VC) = \{\tet \in \Ns^{\circ} \colon \VC_{\tet} \neq \{0\}\}$$ of $\Ns^{\circ}$. In the case where $N$ is a normal subgroup of $G$, then $\Ns$ is a normal subgroup of $\Gs$, and hence $\Gs$ acts by conjugation on $\Ns^{\circ}$. Then, $\VC_{\tet}$ is a smooth $\C[C_{\Gs}(\tet)]$-module which satisfies $x\overline{v} = \tet(x)\overline{v}$ for all $x \in \Ns$ and all $\overline{v} \in \VC_{\tet}$. In this situation, the following auxiliary result will be important for us. For any group $G$, we will denote by $\overline{[G,G]}$ the closure of the commutator subgroup $[G,G]$ of $G$.

\begin{lemma} \label{spec1}
Let $\VC$ be an arbitrary smooth $\CGs$-module, and let $N$ be a $\sig$-invariant normal subgroup of $G$. Then, $\spec_{\Ns}(\VC)$ is a $\Gs$-invariant subset of $\Ns^{\circ}$ and $$\VC_{0} = \bigcap_{\tet \in \spec_{\Ns}(\VC)} \VC(\tet)$$ is a $\CGs$-submodule of $\VC$. In particular, if $\VC$ is irreducible and $\spec_{\Ns}(\VC)$ is non-empty, then $\overline{[\Ns,\Ns]}$ acts trivially on $\VC$, so that $\VC$ becomes naturally as an irreducible smooth $\C\big[ \Gs \slash \overline{[\Ns,\Ns]}\big]$-module.
\end{lemma}

\begin{proof}
For the first assertion, it is enough to observe that $\VC(\tet^{g}) = g^{-1} \VC(\tet)$ for all $\tet \in \Ns^{\circ}$ and all $g \in \Gs$. For the second assertion, we note that $\VC_{0} \neq \VC$, and so $\VC_{0} = \{0\}$. Therefore, the natural linear map $\VC \rightarrow \prod_{\tet \in \spec_{\Ns}(\VC)} \VC_{\tet}$ is injective, and thus $\overline{[\Ns,\Ns]}$ acts trivially on $\VC$.
\end{proof}

As a consequence of \cite[Corollaire~1 au Th\'eor\`eme~3]{Rodier1977a}, we obtain the following result.

\begin{lemma} \label{orbit2}
Let $N$ be a $\sig$-inariant normal subgroup of $G$, let $\VC$ be an irreducible smooth $\CGs$-module, and let $\tet\in \Ns^{\circ}$ be such that $\VC_{\tet} \neq \{0\}$. If the $\Gs$-orbit $\tet^{\Gs}$ is a locally closed subset of $\Ns^{\circ}$, then $\spec_{Q}(\VC) = \tet^{G}$. Moreover, $\VC_{\tet}$ is an irreducible smooth $C_{\Gs}(\tet)$-module and $\VC \cong \CInd^{\Gs}_{C_{\Gs}(\tet)}(\VC_{\tet})$.
\end{lemma}

\begin{proof}
We consider the abelianisation $\Ns^{\ab} = \Ns\slash \overline{[\Ns,\Ns]}$ of $\Ns$, and note that the group $\Gs$ acts naturally by conjugation on $\Ns^{\ab}$. By the previous lemma, $\VC$ has a structure of irreducible smooth $\C[\Ns^{\ab}]$-module. Moreover, $\tet$ may be identified with a smooth character of $\Ns^{\ab}$, and it is clear that the $\Gs$-orbit $\tet^{\Gs}$ is a locally closed subset of $(\Ns^{\ab})^{\circ} = \Ns^{\circ}$. Therefore, it follows from \cite[Corollaire~1 au Th\'eor\`eme~3]{Rodier1977a} that $\spec_{\Ns^{\ab}}(\VC) = \tet^{\Gs}$, and that $\VC_{\tet}$ is an irreducible smooth $C_{\Gs}(\tet)$-module. Furthermore, \cite[Corollaire~2 au Th\'eor\`eme~3]{Rodier1977a} implies that $\VC \cong \CInd^{\Gs}_{C_{\Gs}(\tet)}(\VC_{\tet})$, as required.
\end{proof}

We are now able to prove our main result.

\begin{proof}[Proof of Theorem \ref{main1}]
We proceed by induction on $\dim J$, the result being obvious if $\dim J =1$. Therefore, we assume that $\dim J \geq 2$, and that the result is true whenever $J'$ is a subalgebra of $J$ with $\dim J' \lneqq \dim J$.

Let $\VC$ be an arbitrary irreducible smooth $\CGs$-module. The case where $\dim \VC = 1$ is trivial, and thus we assume that $\dim \VC\geq 2$. As justified in \refs{step}, we may choose the smallest positive integer $m \geq 2$ for which there exists a $\Gs$-invariant smooth character $\vsig\in \Ms^{\circ}$, where $M = 1 + J^{m}$, such that $g\cdot v = \vsig(g)v$ for all $g \in M$ and all $v \in \VC$. Furthermore, there exists an ideal $L$ of $J$ satisfying $$J^{m} \subseteq L \subseteq J^{m-1},\quad \dim(C_{L}(\sig) \slash C_{J^{m}}(\sig))=1,\quad \text{and} \quad \vsig([\Gs,\Ns]) \neq 1$$ where $N = 1+L$.

Since we are assuming by induction that every irreducible smooth $\CNs$-module is admissible, it follows from \cite[Corollary 4.8]{Boyarchenko2011a} that $\Res^{\Gs}_{\Ns}(\VC)$ has an irreducible quotient $\VC'$. Since $[\Ns,\Ns]\subseteq \ker(\vsig)$ (by \refp{invextension}), Schur's lemma implies that $\VC'$ is one-dimensional, and so it affords a smooth character\linebreak $\tet\in \Ns^{\circ}$. We note that the extreme case where $m = 2$ and $\dim J = \dim J^{2} + 1$ cannot occur; indeed, in this situation, we must have $N = G$ (hence, $\Ns=\Gs$), and thus $\VC'= \VC$ which contradicts the assumption $\dim \VC \geq 2$.

Let $\VC_\tet$ is the largest quotient of $\VC$ where $\Ns$ acts via the character $\tet$. We note that $\VC'$ is a quotient of $\VC_{\tet}$, and thus  $\VC_{\tet} \neq 0$. On the other hand, we recall from \refp{invextension} that the $\Gs$-orbit $\tet^{\Gs} \subseteq \Ns^{\circ}$ of $\tet$ consists of all $\tet' \in \Ns^{\circ}$ which satisfy $\tet'\vert_{\Ms} = \vsig$. In particular, $\tet^{\Gs}$ is a closed subset of $\Ns^{\circ}$, and thus \refl{orbit2} assures that $\VC_\tet$ is an irreducible smooth $\C[C_{\Gs}(\tet)]$-module satisfying $$\VC \cong \CInd^{\Gs}_{C_{\Gs}(\tet)}(\VC_\tet).$$

Since $\Ms$ acts on $\VC$ via $\vsig$, we have $\tet\vert_{\Ms} = \vsig$, and so $C_{\Gs}(\tet) =\Psi(S)$ for some $\ks$-vector subspace $S$ of $\Js$ satisfying $$C_{J^{2}}(\sig) \subseteq S \quad \text{and} \quad \dim S \geq \dim \Js -1$$ (see \refp{invextension}). Since $\tet([\Gs,\Ns]) = \vsig([\Gs,\Ns]) \neq 1$, we have $C_{\Gs} (\tet) \neq \Gs$, and thus $\dim S \lneqq \dim \Js$ as $\ks$-vector spaces. It follows that $\dim S = \dim \Js -1$, and hence $\Js = \ks u \oplus S$ for some $u \in \Js$.

Let $\widehat{S}$ be the $\k$-vector subspace of $J$ spanned by $S$, and define $J_{0} = \widehat{S} + J^{2}$. It is clear that $J_{0}$ is a $\k$-subalgebra of $J$ with $\dim J_{0} = \dim J - 1$. Moreover, $C_{J_{0}}(\sig) = S$. In particular, $G_{0} = 1+J_{0}$ is a proper $\sig$-invariant algebra subgroup of $G$ satisfying $$C_{G_{0}}(\sig) = \Psi(S) = C_{\Gs}(\tet).$$ Therefore, by the induction hypothesis, there exist a $\sig$-invariant algebra subgroup $H$ of $G_{0}$ and a smooth character $\map{\vartheta}{\Hs}{\cx}$ such that $$\VC_{\tet} \cong \CInd^{C_{G_{0}}(\sig)}_{\Hs}(\C_{\vartheta}).$$ Furthermore, the smooth $\C[C_{G_{0}}(\sig)]$-module $\VC_\tet$ is admissible, and thus \cite[Section~5.6, pg.~42]{Vigneras1996a} implies that $$\CInd^{C_{G_{0}}(\sig)}_{\Hs} (\C_{\vartheta}) = \Ind^{C_{G_{0}}(\sig)}_{\Hs}(\C_{\vartheta}).$$ Finally, by transitivity of compact induction, we deduce that $$\VC \cong \CInd^{\Gs}_{C_{G_{0}}(\sig)}(\VC_{\tet}) \cong \CInd^{\Gs}_{C_{G_{0}}(\sig)} \big(\CInd^{C_{G_{0}}(\sig)}_{\Hs}(\C_{\vartheta}) \big) \cong \CInd^{\Gs}_{\Hs}(\C_{\vartheta}).$$

In order to complete the proof, it remains to show that $\VC$ is admissible. However, the admissibility of $$\VC \cong \CInd^{\Gs}_{C_{\Gs}(\tet)}(\VC_{\tet})$$ follows at once from \cite[Th\'eor\`eme~4]{Rodier1977a}, because the $\Gs$-orbit $\tet^{\Gs}$ is closed in $\Ns^{\circ}$ (by \refp{invextension}) and because the smooth $\C[C_{G_{0}}(\sig)]$-module $\VC_\tet$ is admissible. Finally, \cite[Section~5.6, pg.~42]{Vigneras1996a} assures that $$\CInd^{\Gs}_{\Hs} (\C_{\vartheta}) = \Ind^{\Gs}_{\Hs}(\C_{\vartheta}),$$ as required.
\end{proof}

\end{document}